\documentclass{article}
\usepackage[utf8]{inputenc}
\usepackage{amsmath,amsthm,dsfont,pifont,amsfonts,MnSymbol}
\usepackage{verbatim}
\usepackage{faktor}

\title{Perfect maps between submetrizable spaces 
\footnote{2020 Mathematics Subject Classification: 54E35, 54E40, 54C10.}
\footnote{Keywords\textup{:} perfect map, submetrizable space.}}
\author{Vlad Smolin \\ \small{Krasovskii Institute of Mathematics and Mechanics, Ural Federal University,} \\ \small{ Yekaterinburg, Russia} \\ \small{e-mail: SVRusl@yandex.ru}}

\theoremstyle{plain}
\newtheorem{teo}{Theorem}
\newtheorem{lemm}[teo]{Lemma}

\newtheorem{prop}[teo]{Proposition}
\newtheorem{ques}[teo]{Question}

\theoremstyle{definition}
\newtheorem{deff}[teo]{Definition}

\newtheorem{nota}[teo]{Notation}

\renewcommand{\phi}{\varphi}

\begin{document}

\maketitle
\begin{abstract}
We investigate a question posed by Chen \cite{Chen}: if $X$ and $Y$ are paracompact submetrizable spaces and $f:X\to Y$ is a perfect map, can $X$ and $Y$ be submetrized by metrics $\rho$ and $d$ respectively such that $f$ remains perfect with respect to the induced topologies?
\end{abstract}

\section{Introduction}

A space X is submetrizable if X has a weaker metrizable topology, or, equivalently, if X can be mapped onto a metrizable topological space by a continuous one-to-one map.

In \cite{Chen} Chen proved that if $\langle X, \tau_X \rangle$ is a paracompact $\sigma$-space and $f:\langle X, \tau_X \rangle \to \langle Y, \tau_Y \rangle$ is a perfect map, then both $\langle X, \tau_X \rangle$ and $\langle Y, \tau_Y \rangle$ can be submetrized by metric $\rho$ and $d$ respectively such that $f:\langle X, \tau_\rho \rangle \to \langle Y, \tau_d \rangle$ is a perfect map. He also posed the following question:
\begin{ques}
    Let both $\langle X, \tau_X \rangle$ and $\langle Y, \tau_Y \rangle$ be paracompact submetrizable spaces and $f:\langle X, \tau_X \rangle \to \langle Y, \tau_Y \rangle$ be a perfect map. Then can $\langle X, \tau_X \rangle$ and $\langle Y, \tau_Y\rangle$ be submetrized by metric $\rho$ and $d$  respectively such that $f:\langle X, \tau_\rho \rangle \to \langle Y, \tau_d \rangle$?
\end{ques}

We give a negative answer to this question and show that the answer is positive if we additionally assume that $\langle Y, \tau_Y \rangle$ is perfectly normal.

\section{Notation and terminology}

\begin{nota} ...
    \begin{itemize}
        \item[\ding{46}\,] $\mathbb{P} \coloneq $ the set of irrational numbers;
        \item[\ding{46}\,] $\mathbb{Q} \coloneq $ the set of rational numbers;
        \item[\ding{46}\,] $\mathbb{M} \coloneq $ the Michael line, i.e. the set $\mathbb{R}$ with the topology generated by the base $\{U\cup K\ \colon\ U \in \tau_{\mathbb{R}} \text{ and } K \subseteq \mathbb{P}\}$ (for details see page 77 in \cite{enc}).
    \end{itemize}
\end{nota}

\begin{nota} Let $\langle X, \tau \rangle$ be a topological space. Then
    \begin{itemize}
        \item[\ding{46}\,] a subset of $X$ is called {\it cozero} if it is equal to $f^{-1}\big [(0,1]\big ]$ for some continuous function $f\colon \langle X, \tau \rangle \to [0,1]$; 
        \item[\ding{46}\,] if $A \subseteq X$, then $\tau(A) \coloneq \{U \in \tau:A\subseteq U\}$;
        \item[\ding{46}\,] if $A \subseteq X$, then $\mathbf{Cl}_{\tau}(A) \coloneq $ the {\it closure} of $A$ in $X$;
\item[\ding{46}\,] if $A \subseteq X$, then $\tau \upharpoonright A \coloneq \{U \cap A\ \colon\ U \in \tau \} =$ the subspace topology on $A$. Sometimes we will write $\langle A, \tau \rangle$ instead of $\langle A, \tau \upharpoonright A \rangle$;
        \item[\ding{46}\,] $\tau\downarrow_{metr}\ \coloneq \{\sigma \subseteq \tau\ \colon\ \sigma \text{ is a topology and } \langle X, \sigma\rangle \text{ is metrizable}\}$;
        \item[\ding{46}\,] if $\mathcal{F}$ is a set and $n\in\omega$, then $[\mathcal{F}]^{n}\coloneq\{F\subseteq\mathcal{F}:|F|=n\}$;
        \item[\ding{46}\,] $[\mathcal{F}]^{<\omega}\coloneq\bigcup_{n\in\omega}[\mathcal{F}]^n$.
    \end{itemize}
\end{nota}

\begin{deff}
    Let both $\langle X, \tau_X \rangle$ and $\langle Y, \tau_Y \rangle$ be submetrizable spaces, and let $f:\langle X, \tau_X \rangle \to \langle Y, \tau_Y \rangle$ be a perfect map. We say that the triple $\langle f, \langle X, \tau_X \rangle, \langle Y, \tau_Y \rangle \rangle$ is (perfect)-submetrizable, if there exist $\sigma_X \in \tau_X\downarrow_{metr}$ and $\sigma_Y \in \tau_Y\downarrow_{metr}$ such that $f:\langle X, \sigma_X \rangle \to \langle Y, \sigma_Y \rangle$ is perfect.
\end{deff}

\section{Positive result}

\begin{prop}[Lemma 3.1 in \cite{oka}] \label{Oka_lemma}
    Let $X$ be a submetrizable space, and let $\mathcal{U}$ be a $\sigma$-discrete collection of cozero sets of $X$. Then there exist a metric space $M$ and a one-to-one map $f\colon X \to M$ such that $f[U]$ is an open set of $M$ for every $U \in \mathcal{U}$.  
\end{prop}

\begin{lemm} \label{analog_of_Oka}
    Let $\langle X, \tau \rangle$ be a collectionwise normal submetrizable space and let $\mathcal{F}$ be a locally finite collection of closed subsets of $\langle X, \tau \rangle$ such that $\forall \mathcal{F}^\prime\subseteq\mathcal{F}\ [\bigcup \mathcal{F}^\prime \text{ is } G_\delta \text{ in } \langle X, \tau \rangle]$. Then there exists $\sigma \in \tau\downarrow_{metr}$ such that $\mathcal{F}$ is a locally finite collection of closed subsets of $\langle X, \sigma \rangle$.
\end{lemm}
\begin{proof}
    For every $F \in [\mathcal{F}]^{<\omega}$ let 
    $$
    \mathbf{A}(F) \coloneq \bigcap F \setminus \{C\in\mathcal{F}:C \not \in F\}.
    $$
    Note that $\mathbf{A}(F)$ is $F_\sigma$, so let 

    \begin{equation} \label{AF_eq_AFi}
        \mathbf{A}(F) = \bigcup_{i \in \omega} \mathbf{A}^i(F),
    \end{equation}
    where $\mathbf{A}^i(F)$ is closed. By the standard argument it follows that for all $n \in \omega$ 
    $$\{\mathbf{A}(F):F\in[\mathcal{F}]^n\} \text{ is locally finite and disjoint},$$
    so for all $n \in \omega$ and $i \in \omega$
    $$
        \{\mathbf{A}^i(F):F\in[\mathcal{F}]^n\} \text{ is a discrete family of closed sets}.
    $$
    Since $\langle X, \tau \rangle$ is collectionwise normal, it follows that for all $n\in\omega$ and $i\in\omega$ there exists a discrete family of open sets $\{\mathbf{W}^i(F):F\in[\mathcal{F}]^n\}$ such that $\mathbf{A}^i(F)\subseteq \mathbf{W}^i(F)$ for all $F\in[\mathcal{F}]^n$. Now for every $F \in [\mathcal{F}]^{<\omega}$ and $i\in\omega$ let
    $$
        \mathbf{AW}^i(F) \coloneq \mathbf{W}^i(F) \setminus \{C\in\mathcal{F}:C \not \in F\}.
    $$
    From the definition of $\mathbf{A}(F)$ it follows that $\mathbf{A}^i(F)\subseteq\mathbf{AW}^i(F)$; since $\langle X, \tau \rangle$ is normal, there exists a cozero set $\mathbf{U}^i(F)$ such that 
    \begin{equation} \label{AiF_sub_UiF}
        \mathbf{A}^i(F) \subseteq \mathbf{U}^i(F) \subseteq \mathbf{AW}^i(F).
    \end{equation} 
    Finally, for every $n\in\omega$ and $i\in\omega$ let
    $$
        \mathcal{U}^n_i \coloneq \{\mathbf{U}^i(F):F\in[\mathcal{F}]^n\}.
    $$
    Note that $\mathcal{U}^n_i$ is a discrete family of cozero sets, then from Proposition \ref{Oka_lemma} it follows that there exists $\sigma \in \tau\downarrow_{metr}$ such that $\mathcal{U}^n_i \subseteq \sigma$ for all $n\in\omega$ and $i\in\omega$.

    Note that for all $i\in\omega$, for all $F\in[\mathcal{F}]^{<\omega}$, and for all $C\in \mathcal{F}\setminus F$, we have $\mathbf{U}^{i}(F)\cap C=\varnothing$. Now, using (\ref{AF_eq_AFi}) and the first inclusion in (\ref{AiF_sub_UiF}), it is easy to deduce that $\mathcal{F}$ is a locally finite collection of closed subsets of $\langle X, \sigma \rangle$.
\end{proof}

\begin{teo}
    Let $\langle Y, \tau_Y \rangle$ be a submetrizable, collectionwise normal and perfect space; if $\langle X, \tau_X \rangle$ is a submetrizable space and $f\colon \langle X, \tau_X \rangle \to \langle Y, \tau_Y \rangle$ is a perfect map, then the triple $\langle f, \langle X, \tau_X \rangle, \langle Y, \tau_Y \rangle \rangle$ is (perfect)-submetrizable.
\end{teo}

\begin{proof}
    Fix $\sigma_X\in \tau_X\downarrow_{metr}$ and a metric $\rho$ which generates $\sigma_X$. Consider base $\mathcal{B} = \bigcup_{i\in\omega}\mathcal{B}_i$ for $\langle X, \sigma_X\rangle$ such that
    \begin{itemize}
        \item[(a)] $\mathcal{B}_m$ is locally finite in $\langle X, \sigma_X\rangle$ for all $m\in\omega$;
        \item[(b)] $\mathbf{diam}_\rho(U) < \frac{1}{m}$ for all $U \in \mathcal{B}_m$;
        \item[(c)] $\mathcal{B}_m$ is a cover of $X$ for all $m\in\omega$.
    \end{itemize}

    For every $m \in \omega$ let $\mathbf{Cl}_{\sigma_X}(\mathcal{B}_m) \coloneq \{\mathbf{Cl}_{\sigma_X}(U): U\in\mathcal{B}_m\}$. Then $\mathbf{Cl}_{\sigma_X}(\mathcal{B}_m)$ is a locally finite family of closed sets in $\langle X, \sigma_X\rangle$, and so in $\langle X, \tau_X\rangle$. Since $f$ is perfect, $f[\mathbf{Cl}_{\sigma_X}(\mathcal{B}_m)] = \{f[C]:C\in \mathbf{Cl}_{\sigma_X}(\mathcal{B}_m)\}$ is a locally finite family of closed sets in $\langle Y, \tau_Y\rangle$, then from Lemma \ref{analog_of_Oka} it follows that there exists $\sigma_Y(m)\in \tau_Y\downarrow_{metr}$ such that
    $$
        f[\mathbf{Cl}_{\sigma_X}(\mathcal{B}_m)] \text{ is a locally finite collection of closed subsets of }\langle Y, \sigma_Y(m)\rangle.
    $$

    Since $\sigma_Y(m)$ is perfectly normal and has a $\sigma$-discrete base for all $m \in \omega$, from Proposition \ref{Oka_lemma} it follows that there exists $\sigma_Y\in \tau_Y\downarrow_{metr}$ such that $\forall m\in\omega[\sigma_Y(m)\subseteq\sigma_Y]$. So we have
    \begin{equation} \label{f_Bn_loc_fin}
        \forall m\big[f[\mathbf{Cl}_{\sigma_X}(\mathcal{B}_m)] \text{ is a closed locally finite collection in }\langle Y, \sigma_Y\rangle\big].
    \end{equation}
    
    Let $\gamma \coloneq$ the topology on $X$ generated by $\sigma_X\cup f^{-1}[\sigma_Y]$ as a subbase. Let us prove that $\langle X, \gamma\rangle$ is metrizable. Consider base $\mathcal{V} = \bigcup_{i \in \omega} \mathcal{V}_i$ for $\langle Y, \sigma_Y\rangle$ where $\mathcal{V}_i$ is locally finite. Firstly, note that 
    \begin{equation} \label{base_for_gamma}
        \{U\cap f^{-1}[V]:U\in\sigma_X \text{ and } V \in \sigma_Y\} \text{ is a base for } \langle X, \gamma\rangle.
    \end{equation}
    Then $\{U\cap f^{-1}[V]:U\in\mathcal{B}\text{ and } V \in \mathcal{V}\}$ is a base for $\langle X, \gamma\rangle$ too. Since 
    $$
        \{U\cap f^{-1}[V]:U\in\mathcal{B}\text{ and } V \in \mathcal{V}\} = \bigcup_{i,j\in\omega}\{U\cap f^{-1}[V]:U\in\mathcal{B}_i\text{ and } V \in \mathcal{V}_j\}
    $$
    is $\sigma$-locally finite, from the Bing-Nagata-Smirnov metrization theorem (It is easy to check that this space is regular) it follows that $\langle X, \gamma\rangle$ is metrizable. 
    
    The map $f:\langle X, \gamma \rangle \to \langle Y , \sigma_Y \rangle$ is obviously continuous and compact, let us prove that it is closed.
    
    Fix $y \in Y$ and $U \in \gamma(f^{-1}(y))$, we need to prove that there exists $H \in \sigma_Y(y)$ such that $f^{-1}[H] \subseteq U$. From (\ref{base_for_gamma}) it follows that for every $x \in f^{-1}(y)$ there exists $U_x \in \gamma(x)$, such that $U_x = V_x \cap f^{-1}[W_x] \subseteq U$, where $V_x \in \sigma_X(x)$ and $W_x \in \sigma_Y(y)$. Since $f^{-1}(y)$ is compact in $\gamma$, then there exist $x_0, \dots, x_{n-1} \in f^{-1}(y)$ such that $f^{-1}(y) \subseteq U_{x_0} \cup \dots \cup U_{x_{n-1}}$ and $f^{-1}(y)\cap U_{x_i} \neq \varnothing$ for all $i \in n$. Note that $f^{-1}(y) \subseteq f^{-1}[W_{x_i}]$ for all $i \in n$. Let 
    $$
        W \coloneq \bigcap_{i\in n} f^{-1}[W_{x_i}],
    $$
    then 
    $$
        f^{-1}(y) \subseteq W\cap(V_{x_0}\cup\dots\cup V_{x_{n-1}}) \subseteq U_{x_0} \cup \dots \cup U_{x_{n-1}} \subseteq U.
    $$
    Since $f^{-1}(y)$ is compact in $\sigma_X$, it follows that $\rho(f^{-1}(y), X\setminus \bigcup_{i\in n} V_{x_i}) > 0$, so fix $m \in \omega$ such that 
    \begin{equation} \label{1_m}
        \frac{1}{m} < \rho(f^{-1}(y), X\setminus \bigcup_{i\in n} V_{x_i}).
    \end{equation}
    From (b) it follows that
    \begin{equation} \label{B_m}
        \forall V \in \mathcal{B}_m\ [\text{if } \mathbf{Cl}_{\sigma_X}(V) \cap f^{-1}(y) \not = \varnothing \text{ then } \mathbf{Cl}_{\sigma_X}(V) \subseteq \bigcup_{i\in n}V_{x_i}].
    \end{equation}
    Let
    $$
        \mathcal{F} \coloneq \{f[\mathbf{Cl}_{\sigma_X}(V)] : V\in \mathcal{B}_m \text{ and } \mathbf{Cl}_{\sigma_X}(V) \cap f^{-1}(y) = \varnothing\}.
    $$
    Let $G\coloneq Y \setminus \bigcup\mathcal{F}$. Since $\mathcal{F} \subseteq f[\mathbf{Cl}_{\sigma_X}(\mathcal{B}_m)]$, from (\ref{f_Bn_loc_fin}) it follows that $G \in \sigma_Y$. We prove that 
    $$
        f^{-1}[G] \subseteq \bigcup_{i\in n}V_{x_i}.
    $$
    Suppose that there exists $z \in G$ such that $f^{-1}(z) \not \subseteq \bigcup_{i\in n}V_{x_i}$. Take $x \in f^{-1}(z) \setminus \bigcup_{i\in n}V_{x_i}$, from (\ref{1_m}) it follows that $\rho(f^{-1}(y), x) > \frac{1}{m}$. Now take $V \in \mathcal{B}_m$ such that $x \in V$, from (\ref{B_m}) it follows that $\mathbf{Cl}_{\sigma_X}(V) \cap f^{-1}(y) = \varnothing$, therefore $f[\mathbf{Cl}_{\sigma_X}(V)] \in \mathcal{F}$. Then $z = f(x) \in \bigcup \mathcal{F}$, a contradiction.

    Now we see that $G \cap \bigcap_{i \in n} W_{x_i} \in \sigma_Y(y)$ and 
    $$
        f^{-1}[G \cap \bigcap_{i \in n} W_{x_i}] \subseteq W\cap(V_{x_0}\cup\dots\cup V_{x_{n-1}}) \subseteq U_{x_0} \cup \dots \cup U_{x_{n-1}} \subseteq U.
    $$
\end{proof}

\section{Counterexample}

\begin{lemm} \label{GdeltaR_GdeltaM}
    For all $B\subseteq \mathbb{R}$ and for all $A\subseteq \mathbb{Q}$ $A$ is $G_\delta$ in $\langle B\cup Q, \tau_{\mathbb{R}}\rangle$ if and only if $A$ is $G_\delta$ in $\langle B\cup Q, \tau_{\mathbb{M}}\rangle$.\qed
\end{lemm}

\begin{lemm} \label{QisnotGdelta}
    Let $B\subset\mathbb{R}$ be a Bernstein set. Then for any $a<b \in \mathbb{R}$, the set $(a,b)\cap\mathbb{Q}$ is not $G_\delta$ in $\langle \mathbb{Q} \cup B, \tau_{\mathbb{R}}\rangle$.
\end{lemm}

\begin{proof}
Suppose, for contradiction, that there exist $a<b$ such that $(a,b)\cap\mathbb{Q}$ is $G_\delta$ in $\langle \mathbb{Q} \cup B, \tau_{\mathbb{R}}\rangle$. Then
\[
(a,b)\cap\mathbb{Q} = \bigcap_{n\in\omega} U_n,
\]
where each $U_n$ is open in $\langle \mathbb{Q} \cup B, \tau_{\mathbb{R}}\rangle$. Write $U_n = (\mathbb{Q} \cup B) \cap V_n$ with $V_n\in\tau_{\mathbb{R}}$.
Let
\[
K \coloneq \bigcap_{n\in\omega} V_n \;\cap\; (a,b).
\]
Then $K$ is $G_\delta$ in $\mathbb{R}$. Moreover: $(a,b)\cap\mathbb{Q}\subseteq K \subseteq(a,b)\setminus B$. So $K$ is a dense $G_\delta$ set in $(a,b)$. Consequently, there exists an uncountable and closed set $P\subset K$.

But $P\subset K \subset \mathbb{R}\setminus B$, i.e. $P\cap B = \varnothing$. This contradicts the definition of a Bernstein set: every uncountable closed subset of $\mathbb{R}$ must intersect both $B$ and its complement. Hence our assumption was false.
\end{proof}

\begin{teo}
    There exist hereditary paracompact first countable spaces $X$ and $Y$ and a perfect map $f: X \to Y$ such that the triple $\langle f, X, Y\rangle$ is not (perfect)-submetrizable.
\end{teo}

\begin{proof}
    Let $\mathbb{P}=B_0\cup B_1$, where $B_0$ and $B_1$ are disjoint Bernstein subsets of $\mathbb{R}$.
    Let ($\mathsf{D}(2)$ is the set $2$ with the discrete topology)
    $$X \coloneq (\mathsf{D}(2) \times \mathbb{M}) \setminus (\{0\}\times B_1 \cup \{1\} \times B_0)$$
    and
    $$
        \pi \coloneq \text{ the projection of $\mathsf{D}(2) \times \mathbb{M}$ onto $\mathbb{M}$}.
    $$
    Since $\mathsf{D}(2)$ is compact and $\{0\} \times B_1 \cup \{1\} \times B_0$ is open in $\mathsf{D}(2) \times \mathbb{M}$, from \cite[Theorem 3.7.1]{eng} and \cite[Proposition 3.7.6]{eng} it follows that $\pi\upharpoonright X \colon X\to \mathbb{M}$ is perfect and onto. Let $f\coloneq \pi\upharpoonright X$. Now we prove that the triple $\langle f, X, \mathbb{M}\rangle$ is not (perfect)-submetrizable.

    Suppose that there exist $\sigma_X \in \tau_X\downarrow_{metr}$ and $\sigma_Y \in \tau_{\mathbb{M}}\downarrow_{metr}$ such that $f: \langle X, \sigma_X\rangle \to \langle \mathbb{R}, \sigma_Y \rangle$ is continuous and perfect. Take $q \in \mathbb{Q}$ and consider points $q_0 = \langle 0, q\rangle$ and $q_1 = \langle 1, q\rangle$ of $X$. Take $U_0 \in \sigma_X(q_0)$ such that $q_1 \not \in \mathsf{Cl}_{\sigma_X}(U_0)$. Then there exist $a < b \in \mathbb{R}$ such that 
    $$
        q_0 \in \big(\{0\}\times(a,b)\big) \cap X \subseteq U_0
    $$
    and
    $$
        q_1 \in \big( \{1\}\times(a,b) \big) \cap X \subseteq X \setminus \mathsf{Cl}_{\sigma_X}(U_0).
    $$
    Since $f: \langle X, \sigma_X\rangle \to \langle \mathbb{R}, \sigma_Y \rangle$ is perfect, then $f[\mathsf{Cl}_{\sigma_X}(U_0)]$ is closed in $\langle \mathbb{R}, \sigma_Y \rangle$, and so it is $G_\delta$ in $\langle \mathbb{R}, \sigma_Y \rangle$. It follows that $f^{-1}[f[\mathsf{Cl}_{\sigma_X}(U_0)]]$ is $G_\delta$ in $\langle X, \sigma_X\rangle$, and so in $X$. From the definition of $f$ and $X$ it follows that
    $$
        f^{-1}[f[\mathsf{Cl}_{\sigma_X}(U_0)]] \cap (\big(\{1\}\times(a,b)\big)\cap X) = \{1\} \times ((a,b)\cap\mathbb{Q}).
    $$
    It follows that $\{1\} \times ((a,b)\cap\mathbb{Q})$ is $G_\delta$ in $X$, and so $(a,b)\cap\mathbb{Q}$ is $G_\delta$ in $\langle \mathbb{Q}\cup B_1, \tau_\mathbb{M} \rangle$. Now from Lemma \ref{GdeltaR_GdeltaM} and Lemma \ref{QisnotGdelta} we obtain a contradiction.
\end{proof}

{\bf Acknowledgement} The work was performed as part of research conducted in the Ural Mathematical Center with the financial support
of the Ministry of Science and Higher Education of the Russian
Federation (Agreement number № 075-02-2026-737).


\bigskip
\begin{thebibliography}{1}

\bibitem{eng} Engelking, R. (1989). General topology. Sigma series in pure mathematics, 6.

\bibitem{enc} Hart, K., Nagata, J., Vaughan, J. {\it Encyclopedia of General Topology.} (Elsevier, 2003)

\bibitem{Lin2} Lin, S., Yun, Z., {\it Generalized Metric Spaces and Mappings}, Science Press, Atlantis Press, 2017.

\bibitem{oka} Oka, S. (1983). Dimension of Stratifiable Spaces. Transactions of the American Mathematical Society, 275(1), 231–243. https://doi.org/10.2307/1999015

\bibitem{Chen} Chen, H. (2011). Perfect images of submetric spaces. In Topology Proceedings (Vol. 37, pp. 145-153).
\end{thebibliography}
\end{document}